%% file: 20240925-CorniMagnani.tex
\documentclass[12pt, a4paper]{amsart}
\usepackage[english]{babel}
 \usepackage[utf8x]{inputenc} 
\usepackage{amsmath}
\usepackage{cmll}
\usepackage{amssymb}
\usepackage{amsthm}
\usepackage{xcolor}
\usepackage{bbm}
\usepackage{mathrsfs}
\usepackage[all]{xy}
\usepackage{indentfirst}
\usepackage{fancyhdr}
\usepackage{newlfont}
\usepackage{setspace}
\usepackage{verbatim}
\usepackage{hyperref} 
\usepackage{bbm}

\usepackage[a4paper,top=3.2cm,bottom=3.3cm, left=2.8cm,right=2.8cm]{geometry}

\include{def-Area}

\newcommand{\bV}{\mathbf{V}}

\newcommand{\bW}{\mathbf{W}}

\newcommand{\IL}{\mathcal{I}\hskip0.0001mm\mathcal{L}}

\newtheorem{teo}{Theorem}[section]

\newtheorem{cor}[teo]{Corollary}
\newtheorem{lem}[teo]{Lemma}

\theoremstyle{definition}

\newtheorem{deff}{Definition}[section]

\newtheorem{Remark}[teo]{Remark}

\allowdisplaybreaks

\begin{document}

\title{Symmetry results for the area formula in homogeneous groups}

\author{Francesca Corni}
\address{Francesca Corni: Dipartimento di Matematica\\ Universit\`a di Bologna\\ Piazza di Porta S.Donato 5\\ 40126, Bologna, Italy}
\email{francesca.corni3@unibo.it}
\author{Valentino Magnani}
\address{Valentino Magnani: Dipartimento di Matematica\\ Universit\`a di Pisa\\ Largo Bruno Pontecorvo 5\\ 56127, Pisa, Italy}
\email{valentino.magnani@unipi.it}

\begin{abstract}
We prove that if the shape of the metric unit ball in a homogeneous group enjoys a precise symmetry property, then the associated distance
yields the standard form of the area formula. The result applies to some classes of smooth and nonsmooth submanifolds. We finally prove the equality between spherical measure and centered Hausdorff measure, under two different geometric conditions on the shape of the metric unit ball.
\end{abstract}

\thanks{F.C. is partially supported by INDAM-GNAMPA-2024 project: {\it Free boundary problems to degenerate, nonlinear, local and nonlocal, operators in noncommutative structures}. V.M. is partially supported by the APRISE - {\em Analysis and Probability in Science} project, funded by the University of Pisa, grant PRA 2022 85, by PRIN 2022PJ9EFL {\em Geometric Measure Theory: Structure of Singular Measures, Regularity Theory and Applications in the Calculus of Variations}, funded by the European Union--NextGenerationEU", CUP:E53D23005860006, and the MIUR Excellence Department Project awarded to the Department of Mathematics, University of Pisa, CUP I57G22000700001.}


\date{\today}


\maketitle

\tableofcontents

\pagebreak

\section{Introduction}

The notion of area is a basic concept, that lies at the foundations of several branches of Mathematics. 
Computing the area of a submanifold is an elementary fact, until we wonder which subsets we are considering and which notion of area we are using. Clearly, the use of a specific notion of surface area also depends on the applications. These questions were deeply studied in the first half of the twentieth century, where different notions of $k$-dimensional surface area were proposed, and the lower semicontinuity was a foremost requirement.

Herbert Federer, among the most influential founders of Geometric Measure Theory, devoted his first works to the concept of surface area, especially the {\em Lebesgue area},  \cite{FedererPhD1944,FedererSAI,FedererSAII,Federer1952bams,Federer1955am}, and two important monographs systematically treated these questions,
\cite{Cesari1956-bk,Rado1948-bk}. Somehow one might conclude that in any ``geometric setting'' where the development of its related Geometric Measure Theory is studied, a suitable notion of surface area is the starting point.

In $n$-dimensional Euclidean space, any set with finite $k$-dimensional Hausdorff measure, with $1\le k<n$, can be decomposed into a disjoint union of a $k$-rectifiable set and a purely $k$-unrectifiable set. 
We have a natural relationship between Hausdorff measure and rectifiability, and indeed a general version of the area formula holds for rectifiable sets in arbitrary metric spaces, \cite{Kir94}. 

An interesting class of sets, called {\em intrinsic graphs}, \cite{FMS14, FranchiSerapioni2016IntrLip, SerraCassano2016,CorMag23pr}, plays the role in homogeneous groups as rectifiable sets, or smooth sets, do in Euclidean spaces. They precisely appear in relation to the theory of sets of finite perimeter in stratified groups, 
\cite{FSSC01,FSSC5,Marchi2014,AKLD2009TangSpGro,Mag31,DonLDMV2022}. 
Although intrinsic graphs enjoy some regularity with respect to the group structure, they may also share some features with ``fractal objects''. 
In fact, it is important to underline that regular intrinsic graphs might have Euclidean Hausdorff dimension that is strictly larger than their topological dimension, therefore they are extremely far from being rectifiable, \cite{KirSer04}. Nonetheless, very recently it has been shown that their spherical measure $\cS^\NN$ can be computed using a suitable notion of Jacobian $J\Phi$ for the (intrinsic) graph mapping $\Phi$, see \cite[Definition~7.1]{CorMag23pr}. 
We start by the following area formula for intrinsic graphs, corresponding to \cite[Theorem~1.2]{CorMag23pr}.

\begin{teo}[Area formula for intrinsic graphs]\label{teo:areaIntro}
	Let $\G$ be a homogeneous group and let $(\W,\V)$ be a couple of complementary subgroups of $\G$. Let $n$ and $\NN$ be the topological and the Hausdorff dimensions of $\W$, respectively. We consider an open set $A \subset \W$ and a mapping $\phi:A \to \V$. We also assume that $\phi$ is intrinsically differentiable at any point of $A$ and that $d\phi:A \to \IL(\W,\V)$ is  continuous.
	Setting $\Sigma=\Phi(A)$, where $\Phi(w)=w\phi(w)$ is the graph map of $\phi$, then for every Borel set $B \subset \Sigma $, we have
	\begin{equation}\label{intro:areageneral}
	\int_B \beta_{d}(\T_x) \ d \mathcal{S}^{\NN}(x)=\int_{\Phi^{-1}(B)} J\Phi(w)\  d \mathcal{H}_{|\cdot|}^{n}(w),
	\end{equation}
	where $\T_x$ is the tangent subgroup to $\Sigma$ at $x$. 
\end{teo}
We refer to \cite{CorMag23pr} and Section~\ref{sect:prelim} for more information about the notions appearing in this introduction and further literature.
In the sequel, the symbol  $\G$ will denote a homogeneous group, if not otherwise stated.

A central notion for the present work is that of {\em spherical factor} $\beta_d(\cdot)$, that is a real function acting on the ``intrinsic tangent spaces'' to the set $\Sigma$.

\begin{deff}[Spherical factor]\label{def:sphericalfactor}
	Let $V\subset\G$ be a linear subspace of dimension $n$. The \emph{spherical factor of a homogeneous distance $d$ with respect to 
		$V$} is the number
	$$ \beta_d (V )= \max_{z \in \mathbb{B}(0,1)} \mathcal{H}^n_{|\cdot|} (V \cap \mathbb{B}(z,1)),$$
	where $\G$ is equipped with a fixed scalar product and the associated
	norm $|\cdot|$.
\end{deff}

We wish to stress that in the previous definition $\cH^n_{|\cdot|}$ is the Euclidean $n$-dimensional Hausdorff measure. Passing from intrinsic graphs to smooth submanifolds of homogeneous groups requires a different area formula, see (1.7) of \cite{Magnani2019Area}. 
An underlying difficulty to obtain this formula is that smooth submanifolds, which are not tangent to a horizontal distribution, need not be rectifiable in the Federer's sense, \cite[3.2.14]{Federer69}, 
using the distance of the group.

The next theorem can be seen as a sort of ``meta area formula'', where the basic conditions that give the formula are assumed.
For notation and definitions of this theorem, we refer to \cite{Magnani2019Area} and Section~\ref{sect:prelim}.

\begin{teo}[Area formula for smooth submanifolds]\label{t:area}
We consider a homogeneous group $\G$ and an $n$-dimensional submanifold $\Sigma\subset\G$ of degree $\NN$, and of class $C^1$. 
The spherical measure $\cS^\NN$ is constructed by a fixed homogeneous distance $d$. Let us assume the following two conditions. 
\begin{enumerate}
\item[I.]
Any $x\in\Sigma$ of maximum degree $\NN$ satisfies the ``upper blow-up'', namely
\[
\theta^\NN(\mu_\Sigma,x)=\beta_d(A_x\Sigma).
\]
\item[II.] 
	The subset of points in $\Sigma$ having degree
	less than $\NN$ is $\cS^\NN$-negligible.
\end{enumerate}
Then for any Borel set $B\subset\Sigma$ we have
\beq\label{eq:areaintroN}
\int_B \beta_d(A_x\Sigma)\, d\cS^\NN(x)=\int_B \|\tau^{\tilde g}_{\Sigma,\NN}(x)\|_g\, d\sigma_{\tilde g}(x).
\eeq
\end{teo}
The proof of \eqref{eq:areaintroN} is an immediate consequence of
the measure-theoretic area formula, \cite[Theorem~11]{Mag30}, see also \cite{LecMag22}, and of the negligibility assumption on the set of points of lower degree. For instance, in \cite[Theorem~1.3]{Magnani2019Area} one can find specific cases where the assumptions of Theorem~\ref{t:area} 
are satisfied, see also the references therein.
To the best of our knowledge, area formulas for the spherical measure of smooth submanifolds first
appeared in the works of  P.~Pansu \cite{Pansu82,Pansu82ineq} in Heisenberg groups, and of J.~Heinonen \cite{Hei1995CalcCG} in general Carnot groups.

We are interested in showing that whenever either Theorem~\ref{teo:areaIntro} or Theorem~\ref{t:area} can be applied and the homogeneous distance has some specific symmetry properties, then 
we obtain the standard forms of the area formula, stated in \eqref{eq:areaintroSym}.
In this sense, we may consider the present work
as a continuation of \cite{Magnani2019Area,Mag22RS,CorMag23pr}, to which our results apply.

The objective of our ``symmetry results'' is to find those conditions for which the spherical factor 
$\beta_d(\cdot)$ becomes a {\em geometric constant}. 
The simplest and well known (commutative) case is that of an $n$-dimensional subspace $V$ of the Euclidean space $\R^\q$, for which $\beta_{d_E}(V)=\omega_n$, $d_E$ is the Euclidean distance and $\beta_{d_E}(\cdot)$ is given
in Definition~\ref{def:sphericalfactor}. 
The constant $\omega_n$ here is the Lebesgue measure 
of the Euclidean unit ball of $\R^n$, that usually appears in the definition of 
the $n$-dimensional Hausdorff measure.

The area formulas \eqref{intro:areageneral} and \eqref{eq:areaintroN} become definitively simpler for those homogeneous distances having constant spherical factor on a family $\cF$ of subspaces which includes all the ``suitable'' tangent spaces to $\Sigma$.
Then the constant spherical factor $\omega_d(\cF)$ can be introduced in the definition of spherical measure 
$\cS_d^\NN=\omega_d(\cF)\cS^\NN$, getting
\beq\label{eq:areaintroSym}
\mathcal{S}^{\NN}_d(\Sigma)=\int_{\Phi^{-1}(\Sigma)} J\Phi(w)\  d \mathcal{H}_{|\cdot|}^{n}(w)\quad\text{and}\quad\cS^\NN_d(\Sigma)= \int_\Sigma\|\tau^{\tilde g}_{\Sigma,\NN}(x)\|_g\, d\sigma_{\tilde g}(x).
\eeq
In Heisenberg groups, the area formulas having the useful form of \eqref{eq:areaintroSym} can be found in several works, 
\cite{FSSC01, FSSC6, CMPSC14, Vit22, Mag22RS, CorMag23}. These two standard forms of the area formula for an intrinsic graph and for a smooth submanifold are a straightforward consequence of \eqref{intro:areageneral} and \eqref{eq:areaintroN}, respec\-tively, where we have a constant spherical factor. 
Thus, we have some motivations to introduce the following notion, see also \cite[Definition~1.2]{Mag22RS}.

\begin{deff}[Rotationally symmetric distance]
	We consider a nonempty class $\cF$ of homogeneous subspaces. A homogeneous distance $d$ on a homogeneous group $\G$ is called \textit{rotationally symmetric with respect to $\cF$}, if the spherical factor $\beta_d(\cdot)$ is a constant function on $\cF$. We denote by $\omega_d(\cF)$ the constant value assumed by the restriction of the function $\beta_d(\cdot)$ to $\cF$.
\end{deff}

We are mainly concerned with {\em multiradial distances}, which represent a special class of homogeneous distances,
see \cite{Magnani2019Area} and \cite{Mag22RS}. We will show that multiradial distances are rotationally symmetric with respect to a large class of subspaces. The next definition weakens the assumptions of 
\cite[Definition~5.1]{Mag22RS}.

\begin{deff}[Multiradial distance]\label{def:multiradialdistance}
	A homogeneous distance $d$ on a
	homogeneous group $\G$ is \textit{multiradial} if there exists a continuous function $\varphi : [0, + \infty)^{\iota} \to  [0,+ \infty)$ that is monotone nondecreasing on each single variable, satisfies 
	$\ph(0)=0$ and the metric unit ball is 
	$$ 
	\B(0,1)=\set{x\in\G:\varphi(|x_1|, \ldots,  |x_{\iota}|)\le 1},
	$$
	where $x_j = P_{H_j}(x)$ for $j=1, \ldots, \iota$ and 
	$P_{H_j}$ are defined in \eqref{eq:PHj}.
	We finally require the ``coercivity condition''
	$\varphi(t_1, \ldots, t_\iota)\to+\infty$ as $|t|\to+\infty$.
\end{deff}

We have removed the assumption that the function $\ph$ defining the profile of the metric unit ball also represents the formula for the distance from the unit element. As a consequence, Definition~\ref{def:multiradialdistance} also includes the distance of \cite[Theorem~2]{HebSik90}, whose metric unit ball is a Euclidean ball of suitably small radius.
Clearly, the rescaled Euclidean distance cannot be a homogeneous distance in any graded group of step higher than one.
Other more common examples of multiradial distances are the Cygan--Kor\'anyi distance in H-type groups, \cite{Cygan81}, and the distance $d_\infty$ of \cite[Section~2.1]{FSSC6}. 

An important technical aspect is a simplified formula for the spherical factor with respect to a multiradial distance, established in Theorem~\ref{T1}.
The main consequence is Theorem~\ref{teo:multirot}, where we prove that multiradial distances are rotationally symmetric with respect to a large family of subspaces $\mathcal{F}_{n_1, n_2, \ldots, n_{\iota}}$, see Definition~\ref{def:F_n1n2..}. 
The first application of Theorem~\ref{teo:multirot} allows us to establish a simpler version of \eqref{intro:areageneral}
for multiradial distances, according to the next theorem, proved in Section~\ref{sect:multiradial}.

\begin{teo}[Area of intrinsic graphs for multiradial distances]\label{T2}
	In the hypotheses of Theorem~\ref{teo:areaIntro}, we also assume that $\cS^\NN$ is constructed by a multiradial distance $d$.
	We set $\Sigma=\Phi(A)$ and define the integers $n_i=\dim(\W \cap H_i)$ for every $i=1, \ldots, \iota$. 
	We denote by $\omega_d(\mathcal{F}_{n_1, \ldots, n_{\iota}})$ the constant spherical factor, due to Theorem~\ref{teo:multirot}. If we set
	$\mathcal{S}_d^\NN=\omega_d(\mathcal{F}_{n_1, \ldots, n_{\iota}}) \mathcal{S}^\NN$, then
	for every Borel set $B \subset \Sigma$ we have
	\begin{equation}\label{eq:SN(B)} 
		\mathcal{S}_d^\NN(B)=\int_{\Phi^{-1}(B)} J\Phi(w)\  d \mathcal{H}_{|\cdot|}^{n} (w).
	\end{equation}
\end{teo}
We emphasize that the regular sets of Theorem~\ref{T2} also include the large class $(\G,\M)$-regular sets of $\G$, introduced in \cite[Definition 3.5]{Mag5} and subsequently studied in \cite{Mag14,JNGV22}. For these sets, a special form of the area formula holds, see \cite[Theorem~1.4]{CorMag23pr}.
As a result, Theorem~\ref{T2} leads us to the next result.
Concerning the notions involved in the next corollary and its proof, we refer the reader to \cite{CorMag23pr} and Section~\ref{sect:multiradial}.

\begin{cor}[Area of level sets for multiradial distances]\label{cor:levelsetmultiradial}
	Let $\Omega \subset \G$ be an open set and let $f \in C^1_h(\Omega, \M)$. Let us define the level set $\Sigma=f^{-1}(0)$ and assume that there exist an open set $\Omega' \subset \Omega$ and a homogeneous subgroup $\V \subset \G$ of topological dimension $\p$ such that
	$J_\V f(y)> 0$
	for any $y \in \Sigma \cap \Omega'$. Let $\W \subset \G$ be a homogeneous subgroup complementary to $\V$ and consider the unique map $\phi:A\to\V$, 
	whose graph mapping 
	$\Phi:A\to\G$ satisfies $\Sigma \cap \Omega' = \Phi(A)$, where $A\subset\W$ is an open set. Let $\bV$ be an orienting unit $\p$-vector of $\V$ and let $\bW$ be an orienting unit $(\q-\p)$-vector of $\W$. 
	
	We assume that $\G$, of topological dimension $\q$, is equipped with a multiradial distance $d$ and in view of Theorem~\ref{teo:multirot}, we set 
	$\mathcal{S}_d^{Q-P}=\omega_d(\mathcal{F}_{n_1, \ldots, n_{\iota}}) \mathcal{S}^{Q-P}$.
	Thus, for every Borel set $B \subset \Sigma \cap \Omega'$, we have
	\begin{equation}
		\mathcal{S}_d^{Q-P}(B)= | \bV \wedge \bW| \int_{\Phi^{-1}(B)} \frac{J_Hf(\Phi(n))}{J_{\V}f(\Phi(n))} \  d \mathcal{H}_{|\cdot|}^{\q-\p} (n),
	\end{equation}
	where $Q$ and $P$ denote the Hausdorff dimensions of
	$\G$ and $\M$, respectively. 
\end{cor}

As already mentioned, Theorem~\ref{teo:multirot} has also consequences for the area formula of smooth submanifolds. Joining such theorem with Theorem~\ref{t:area}, we obtain a ``standard form'' of the area formula for smooth submanifolds.

\begin{teo}[Area of smooth submanifolds for multiradial distances]\label{teo:areasm}
	We consider a homogeneous group $\G$ and an $n$-dimensional submanifold $\Sigma\subset\G$ of degree $\NN$, and of class $C^1$. 
	The spherical measure $\cS^\NN$ is constructed by a multiradial distance $d$. Let us assume that the following three conditions hold. 
	\begin{enumerate}
		\item[I.]
		Any $p\in\Sigma$ of maximum degree $\NN$ satisfies the ``upper blow-up'', namely
		\[
			\theta^\NN(\mu_\Sigma,p)=\beta_d(A_p\Sigma).
		\]
		\item[II.] 
		The subset of points in $\Sigma$ having degree
		less than $\NN$ is $\cS^\NN$-negligible.
		\item[III.]
		We have $A_p\Sigma\in\cF_{n_1,\ldots,n_\iota}$ for each
		homogeneous tangent spaces at a point $p$ of maximum degree. 
	\end{enumerate}
	Then for any Borel set $B\subset\Sigma$ we have
	\beq\label{eq:areaintroNmulti}
	\cS_d^\NN(B)=\int_B \|\tau^{\tilde g}_{\Sigma,\NN}(p)\|_g\, d\sigma_{\tilde g}(p),
	\eeq
	where we have set $\cS^\NN_d=\omega(\cF_{n_1,\ldots,n_\iota})\cS^\NN$,
	and $\omega(\cF_{n_1,\ldots,n_\iota})$ is the constant spherical factor, due Theorem~\ref{teo:multirot}.
\end{teo}
In a few words, whenever an area formula holds and all homogeneous tangent spaces belong to $\cF_{n_1,\ldots,n_\iota}$, then \eqref{eq:areaintroNmulti} holds for multiradial distances.
The previous theorem can be also seen as a tool to obtain the standard area formula for the spherical measure of a smooth submanifold.
We notice that Theorem~\ref{teo:areasm} includes \cite[Theorem~1.3]{Mag22RS}.

The last part of this work is devoted to the relationship between spherical measure and centered Hausdorff measure on subsets of
homogeneous groups. The $\alpha$-dimensional centered Hausdorff measure $\cC^\alpha$, also called {\em covering measure}, is well known in Fractal Geometry, \cite{SaintRayTri88,Edg2007}.  
In the setting of homogeneous groups, it has been first studied in \cite{FSSC15}, where among other things, the equality $\cC^Q=\cS^Q$ was proved in any homogeneous group of Hausdorff dimension $Q$. 
For the multiradial distance $d_\infty$, the authors also proved the equality $\cC^{Q-1}=\cS^{Q-1}$ on one codimensional intrinsic regular sets, and then for $\G$-rectifiable sets, see \cite[Theorem~4.28]{FSSC15}.

Our last result is the extension of the equality between spherical measure and centered Hausdorff measure to higher codimensional intrinsic graphs.

\begin{teo}\label{teo:C=S}
	Let $\Sigma \subset \G$ be an intrinsic graph 
	associated with a couple of complementary subgroups $(\W,\V)$ and
	of Hausdorff dimension $\NN$. 
	Both spherical measure and centered Hausdorff measure are constructed
	by a fixed homogeneous distance $d$ on $\G$. 
	We assume that one of the following two conditions holds.
	\begin{enumerate}
		\item 
		$\Sigma$ is the graph of the mapping $\phi:A \to \V$, where $A \subset \W$ is open, $\phi$ is continuously intrinsically differentiable on $A\subset\W$ and $d$ is multiradial. 
		\item $\Sigma$ is a $(\G,\M)$-regular set of $\G$ and the metric unit ball $\B(0,1)$ of $d$ is convex.
	\end{enumerate} 
Then in any of the two conditions it follows that 	
\[
\mathcal{S}^\NN\res \Sigma=\mathcal{C}^\NN\res\Sigma.
\]
\end{teo}
It is worth to mention that in the previous assumptions the spherical factor is not required to be constant. 
When $\G$ is an Heisenberg group, the previous result includes \cite[Theorem~4.2]{CorMag23} as a special instance.
The condition (1) of Theorem~\ref{teo:C=S} follows from Theorem~\ref{teo:eqSC}
whereas the condition (2) is a consequence of Theorem~\ref{th:CSSQ-P}.
It is a little bit surprising that multiradial distances, without any convexity assumption, 
satisfy the same symmetry condition \eqref{eq:betamultiradial}
of Theorem~\ref{teo:ballconv}, where the metric unit ball of the given homogeneous distance is a convex set.

Finding homogeneous distances that allow for a constant spherical factor is not an easy task in general homogeneous groups. The question is strictly related to the metric and the algebraic structure of the group. For instance, other types of symmetric homogeneous distances are possible, like {\em vertically symmetric distances}, \cite{Mag31,Mag22RS,CorMag23}.
These results confirm that further study is necessary to understand the geometric properties of symmetric distances in homogeneous groups.

\section{Preliminaries and basic facts}\label{sect:prelim}

\subsection{Homogeneous groups and some geometric measures} 
The present section is devoted to the basic notions that will be
used throughout. 
A \textit{graded group} $\G$ of step $\iota$ is a connected, simply connected and nilpotent Lie group, whose Lie algebra is graded of step $\iota$, namely there exists a sequence of subspaces $\cH_j$ with $j\in\N$, such that $\cH_j= \{ 0 \}$ if $j>\iota$, $[\cH_i, \cH_j] \subseteq \cH_{i+j}$
for every $i,j\ge1$, $\cH_\iota \neq \{ 0 \}$ and
$ \mathrm{Lie}(\G)=\cH_1 \oplus \dots \oplus \cH_\iota$, where 
$$[\cH_i,\cH_j]=\text{span} \{ [X,Y]  :  X \in \cH_i, \ Y \in \cH_j\}.$$  
If $[\cH_1, \cH_i] = \cH_{i+1}$ for every $i=1, \dots, \iota-1$, we say that $\G$ is a \textit{stratified group}.

The exponential map $\mathrm{exp}: \mathrm{Lie}(\G) \to \G$ is a global diffeomorphism, hence we are allowed to identify in a standard way $\G$ with $\mathrm{Lie}(\G)$, namely we model a graded group $\G$ as a graded vector space 
\begin{equation}\label{eq:dirsumG}
H_1 \oplus H_2 \oplus \dots \oplus H_\iota,
\end{equation}
endowed with both a Lie group and a Lie algebra structure. 

The group operation on $\G$ is given by the well known Baker--Campbell--Hausdorff formula, in short {\em BCH formula}, see for instance \cite[Section~2.15]{Var84Lie}.
The {\em left translation by an element $x\in\G$} is the analytic diffeomorphism $l_x:\G \to \G$, $l_x(y)=xy$ for every $y \in \G$.
The linear projection with respect to the direct sum \eqref{eq:dirsumG}
is  
\begin{equation}\label{eq:PHj}
P_{H_j} : \G \to H_j, \qquad j=1, \ldots, \iota
\end{equation}
A {\em homogeneous group} is a graded Lie group equipped with 
a one-parameter group of ``dilations'' $\set{\delta_r:r>0}$,
that read on $\Lie(\G)$ have eigenvalues $r^i$ on $H_i$, $i=1,\ldots,\iota$.  We equip a graded group $\G$ with a {\em homogeneous distance}, i.e.\ a distance $d$ on $\G$ such that for every $x,y,z \in \G$ and $r>0$, the conditions 
\[
d(zx,zy)=d(x,y)\quad \text{and}\quad d(\delta_rx, \delta_ry)=r d(x,y)
\]
hold for all $z,x,y\in\G$ and $r>0$. 
We also introduce the associated {\em homogeneous norm} $\| x \|=d(x,0)$ for every $x \in \G$.

A {\em homogeneous subspace} is a linear subspace $V$ of $\G$, which is closed under the action of dilations $\delta_r$. If $V$ is also a subgroup, then we call it a \textit{homogeneous subgroup}.
It can be easily checked that the Hausdorff dimension of $\G$ with respect to $d$ is given by the formula $Q=\sum_{j=1}^{\iota} j \ \mathrm{dim}(H_j)$. Since all homogeneous distances are equivalent to each other, the Hausdorff dimension of $\G$ is independent of the
fixed homogeneous distance. We denote by $\q$ the topological dimension of the homogeneous group $\G$.

Throughout the paper, we assume that $\G$ is equipped with a scalar product $ \langle \cdot, \cdot \rangle$ and we denote by $|\cdot|$ its associated norm. Moreover, we assume that the layers $H_1, \ldots, H_{\iota}$ are orthogonal with respect to $ \langle \cdot, \cdot \rangle$. The linear structure of $\G$ gives a canonical isomorphism between $\G$ and $T_0\G$. Hence, the scalar product $\langle\cdot,\cdot\rangle$ automatically extends to a left invariant Riemannian metric $g$ on $\G$. We denote the norm generated by the inner product on the tangent space $T_x\G$ by $|\cdot|_g$, with $x \in \G$. 

For every $k \in \N$, $1\le k\le \q$, we consider the space $\Lambda_k \G$ of $k$-vectors. The fixed scalar product $\langle \cdot, \cdot \rangle$ naturally extends to a scalar product on $\Lambda_k \G$, so that we have a Hilbert space structure on $\Lambda_k\G$, where the associated norm is still denoted by $| \cdot |$. 
If $V \subset \G$ is a $k$-dimensional subspace, an
\textit{orienting $k$-vector $ \textbf{V} \in \Lambda_k \G\sm\set{0}$  of $V$} is a simple $k$-vector such that $V=\{ v \in \G : \textbf{V} \wedge v=0 \}$.
For $x \in \G$ and $r>0$, it is useful to introduce the following metric balls
$$
\B(x,r)= \{ x \in \G: d(x,0) \leq r \} \qquad B_E(x,r)= \{ x \in \G: |x| < r \}.
$$
In the sequel $\G$ is assumed to be a homogeneous group, if not otherwise stated.

The homogeneous distance of $\G$ gives rise to a natural way to measure subsets with an associated dimension.  
Let $\mathcal{F} \subset \mathcal{P}(\G) $ be a nonempty family of closed subsets of $\G$ and let $\zeta: \mathcal{F} \to [0,+\infty]$ be any function, that is the fixed {\em gauge}. For $\delta>0$, $A \subset \G$, we define
\begin{equation}\label{def:carathdeltameas}
	\phi_{\delta, \zeta}(A)= \inf \set{ \sum_{j=0}^{\infty} \zeta(B_j) \ : \ A \subset \bigcup_{j=0}^\infty B_j , \ \mathrm{diam}(B_j) \leq \delta, \ B_j \in \mathcal{F} }.
\end{equation}
Considering $\phi_\zeta(A)=\sup_{\delta>0}\phi_{\delta,\zeta}(A)$, we have introduced a Borel regular measure $\phi_\zeta$ on the metric space $\G$. 
Given $\alpha \in [0, \infty)$, we set the gauge
\begin{equation}\label{eq:zeta}
\zeta_\alpha(S)=(\diam(S)/2)^{\alpha}
\end{equation}
for every $S\subset \G$. 
If $\mathcal{F}$ coincides with the family $\cF_b$ of closed balls with positive radius and we 
consider $\zeta=\zeta_\alpha|_{\cF_b}$, then the resulting measure $\phi_{\zeta_\alpha}$ is called the \textit{$\alpha$-dimensional spherical measure} and we denote it by $\cS^\alpha$.\\

Following \cite{Edg2007}, or \cite{SaintRayTri88},
we define the \emph{$\alpha$-dimensional centered Hausdorff measure} $\mathcal{C}^{\alpha}$ of a set $A \subset \G$ as 
$$ \mathcal{C}^{\alpha}(A)= \sup_{E \subset A} \mathcal{D}^{\alpha}(E)$$ where
$\mathcal{D}^{\alpha}(E)= \lim_{\delta \to 0+} \mathcal{D}^{\alpha}_{ \delta}(E)$, and for every $\delta \in (0, \infty)$ we have set
$$
\mathcal{D}^{\alpha}_{ \delta}(E)= \inf \  \left\{ \  \sum_{j=0}^{\infty} \zeta_{\alpha}(\B(x_j, r_j)) : E \subset \bigcup_{j=0}^{\infty} \mathbb{B}(x_j, r_j), \ x_j \in E, \ \text{diam}(\mathbb{B}(x_j, r_j)) \leq \delta  \right\}.
$$

Let us consider in \eqref{def:carathdeltameas} the case where $\mathcal{F}$ is the family $\cF_c$ of closed subsets of $\G$. Then we fix $k \in \{1, \dots, q \}$ and define the geometric constant
$\omega_k=\cL^k(\set{x\in\R^k: |x|_{\R^k}\le 1})$,
where $|\cdot|_{\R^k}$ is the Euclidean norm of $\R^k$. 
Considering $\zeta=\zeta_k|_{\cF_c}$, where now the diameter of \eqref{eq:zeta} is considered with respect to the norm $| \cdot|_{\R^k}$ associated with the fixed scalar product $\langle \cdot, \cdot \rangle$ on $\G$, then $\omega_k\phi_{\zeta_k}$ becomes the well known 
$k$-dimensional Hausdorff measure with respect to the Euclidean distance, denoted by $\mathcal{H}^k_{|\cdot|}$.


\subsection{An auxiliary result}
We present two lemmas which may have an independent interest
and which immediately give 
Theorem~\ref{teo:ballconv}. This is the central tool of Section~\ref{sect:spherical_center}. 
The two lemmas concern translations of normal subgroups.

We need first to emphasize the notion of group projection.
Let us fix a homogeneous group $\G$ and choose two homogeneous
subgroups $\W$ and $\V$ with the properties
\begin{equation}\label{def:complemSG}
\W \cap \V = \{0 \}\quad \text{and} \quad \G= \W   \V.
\end{equation}
We call $(\W,\V)$ a \textit{couple of complementary subgroups}. 
With our identification, it is also true that $\W$ and $\V$ are subalgebras of $\G$ such that $ \W \oplus \V=\G$. 
Due to \eqref{def:complemSG}, the {\em group projections}
\beq\label{eq:proj}
\pi_{\W}: \G \to \W, \ \pi_{\W}(wv)=w, \ \pi_{\V}: \G \to \V, \ \pi_{\V}(wv)=v
\eeq
are well defined for every $w\in\W$ and $v\in\V$.

The next lemmas follow from \cite[Lemma~3.1.20, 3.1.21]{CorniPhD} and \cite[Proposition~3.1.22]{CorniPhD}.
Their proof is based on the Baker--Campbell--Hausdorff formula, and  taking into account the grading of the Lie algebra, arguing as in \cite[Lemma~3.3, 3.4, 3.5]{Mag22RS}.
An interesting aspect is that these lemmas do not necessarily require a factorization of $\G$ 
by a couple of complementary subgroups.

\begin{lem}\label{L1}
Let $V,\W \subset \G$ be homogeneous subspaces of a homogeneous
group $\G$, where $\W$ is also a normal subgroup and $\G=V \oplus \W$. 
Then the mapping 
\[
F: V \times \W \to \G,\quad  F(v,w)=vw
\]
is an invertible polynomial function with polynomial inverse 
$T:\G \to V \times \W$. Thus, the group projections $\pi_V$, $\pi_\W$  are defined 
by the formula $T(x)=(\pi_V(x), \pi_{\W}(x))$
for every $x\in\G$. Moreover, the group projection $\pi_V$ is also a linear projection with respect 
to the direct sum $V\oplus\W$. 
\end{lem}

\begin{lem}\label{L2}
If $V,\W \subset \G$ are homogeneous linear subspaces such that $\G=V \oplus \W$
and $\W$ is a normal subgroup. Then for every $v \in V$ and $x\in\G$, we have
$ v+\W=v\W$  and  $ \mathcal{H}^n_{|\cdot|}(B)=  \mathcal{H}^n_{|\cdot|}(l_x(B))
 $ for every measurable set $B \subset \W$.
\end{lem}

As a consequence of the previous lemmas, we obtain the following special form of the 
spherical factor.

\begin{teo}\label{teo:ballconv}
	If $d$ is a homogeneous distance whose metric unit ball $\B(0,1)$ is convex and $\W \subset \G$ is an $m$-dimensional normal subgroup of $\G$, then 
\begin{equation}\label{eq:betaconvex}
\beta_d(\W)=\mathcal{H}^m_{| \cdot |}(\B(0,1) \cap \W).
\end{equation}
\end{teo}

The proof of this theorem follows the same steps of \cite[Theorem 1.4]{Mag22RS},
where the vertical subgroup of \cite{Mag22RS} is replaced by a more general normal subgroup,
and the use of  \cite[Lemmas 3.3, 3.4, 3.5]{Mag22RS} is replaced by Lemmas \ref{L1} and \ref{L2}.

\section{Area of intrinsic graphs by multiradial distances}\label{sect:multiradial}

In this section, we prove that all multiradial distances are rotationally symmetric with respect to a large class
of homogeneous subspaces. We start with the key result.

\begin{teo}\label{T1}
Let $\G$ be a homogeneous group and let $d$ be a multiradial distance. Then for every $n$-dimensional homogeneous subspace $V \subset \G$ the equality
\begin{equation}\label{eq:betamultiradial}
\beta_d(V)=\mathcal{H}^n_{| \cdot |} (V \cap \B(0,1)),
\end{equation}
holds, with $1 \leq n \leq \q-1$.
\end{teo}
\begin{proof}
Let $z \in \B(0,1)$ and let $V=V_1 \oplus \dots \oplus V_{\iota}$, with $V_j \subset H_j$ for every $1 \leq j \leq \iota$, being $V$ an homogeneous subspace. The assumptions on $d$ ensure that
\begin{equation}\label{insieme}
V \cap \B(z,1)=\{ v \in V : \varphi(|P_{H_1}(z^{-1}v)|, \ldots, |P_{H_{\iota}}(z^{-1}v)|) \leq 1 \}.
\end{equation}
Let us denote $z_i=P_{H_i}(z)$ and $v_i=P_{H_i}(v)$, for $i=1, \ldots, \iota$. By the BCH formula, we have that
\begin{equation}\label{BCH}
\begin{split}
z^{-1}v=v&_1-z_1+v_2-z_2+Q_2(v_1,z_1)+v_3-z_3+Q_3(v_1, v_2, z_1, z_2)\\
&+ \ldots +v_{\iota}-z_{\iota}+Q_{\iota}(v_1, \ldots v_{\iota-1}, z_1, \ldots, z_{\iota-1})\\
=&v_1-z_1+\sum_{s=2}^{\iota}\Big( v_s-z_s+Q_s(v_1, \ldots, v_{s-1}, z_1, \ldots, z_{s-1}\Big),
\end{split}
\end{equation}
where we have underlined the fact that for $i=2, \ldots, \iota$, $Q_i: \G \times \G \to H_i$ is a suitable $i$-homogeneous function such that
$Q_i(x,y)$ depends only on the components $P_{H_j}(x), \ P_{H_j}(y)$ such that $j<i$, for every $x,y \in \G$.
Let us introduce
\begin{equation}\label{psi}
\Psi_i(v_1, \ldots, v_{i-1}):=z_i-Q_i(v_1, \ldots, v_{i-1}, z_1, \ldots, z_{i-1}) \in H_i
\end{equation}
for $i=2, \ldots, \iota$.
We notice that we have pointed out only the dependence on $(v_1, \ldots, v_{i-1})$, since $z$ is fixed.
By combining \eqref{insieme}, \eqref{BCH} and \eqref{psi} we get that
\begin{align*}
V \cap B(z,1)&=\{ v \in V : \varphi(|v_1-z_1|, |v_2-z_2+Q_2(v_1,z_1)|, \ldots  \\
 &  \ \ \ \ \ \ \ \ \ \ \ \ \ \ \ \ \ \ \ldots, |v_{\iota}-z_{\iota}+Q_{\iota}(v_1, \ldots v_{\iota-1}, z_1, \ldots, z_{\iota-1})|) < 1 \}\\
&=\{ v \in V : \varphi(|v_1-z_1|, |v_2-\Psi_2(v_1)|, \ldots, |v_{\iota}-\Psi_{\iota}(v_1, \ldots, v_{\iota-1})|) < 1 \}.
\end{align*}
By the coercivity of $\ph$, we can define
\begin{equation}\label{eq:rho_1}
\rho_1:=\sup\{ t \geq 0 : \varphi(t,0, \ldots, 0) < 1 \} \in (0, +\infty).
\end{equation}
Again from the coercivity of $\ph$, for $i=2, \ldots , \iota$, we can 
introduce the following function 
\begin{align}
\rho_i&: T_i  \to (0, \infty), \\
T_i&=\{ (v_1, \ldots, v_{i-1}) \in V_1 \times \ldots \times V_{i-1}: \varphi(|v_1|, \ldots, |v_{i-1}|, 0, \ldots, 0)<1 \} \label{eq:T_i}\\
\rho_i&(v_1, \ldots, v_{i-1}):= \sup \{ t \geq 0 : \varphi(|v_1|, \ldots, |v_{i-1}|, t, 0, \ldots, 0) < 1 \}. \label{eq:rho_i}
\end{align}
By the monotonicity properties of $\varphi$, for every $i=2, \ldots, \iota$ and every $\ell=1, \ldots, i-1$, we have that
 \begin{equation}\label{eq:telle}
 \varphi(|v_1|, \ldots, |v_{\ell}|, 0 , \ldots, 0) \leq \varphi(|v_1|, |v_2|, \ldots, |v_{i-2}|, |v_{i-1}|, 0, \ldots, 0)<1
 \end{equation}
 hence if $(v_1, \ldots, v_{i-1}) \in T_{i}$, then $(v_1, \ldots, v_{\ell-1}) \in T_{\ell}$ for every $\ell \leq i$.

Let us now notice that for every $i=2, \ldots, \iota$, for every fixed $j=1, \ldots, i-1$, if we choose $(u_1, \ldots, u_{j-1}, w_j, u_{j+1}, \ldots u_{i-1}) \in T_i $ and we consider $u_j \in V_j$ such that $|u_j| \leq |w_j|$, then by the non-decreasing monotonicity of $\varphi$ in the $j$-th variable it follows that \begin{equation}\label{eq-dadim}
(u_1, \ldots, u_{j-1}, u_j, u_{j+1}, \ldots u_{i-1}) \in T_i
\end{equation} and by combining the definition of $\rho_i$ and the non-decreasing monotonicity of $\varphi$ we have
\begin{equation}\label{mon1}
 \rho_i(u_1, \ldots, u_{j-1}, w_{j}, u_{j+1}, \ldots, u_{i-1}) \leq \rho_i(u_1, \ldots, u_{j-1}, u_{j}, u_{j+1}, \ldots, u_{i-1}). 
 \end{equation}

Let us now assume that $z=0$. We introduce the Lebesgue measure $\mathcal{L}^n$ on $V$ by choosing an orthonormal basis. Then Fubini's theorem ensures that the following equality holds
\begin{align}\label{zuguale0}
\mathcal{H}^n_{| \cdot |} &(\B(0,1) \cap V)=\mathcal{H}^n_{| \cdot |} (B(0,1) \cap V)=\int_{B_E(0,\rho_1) \cap V_1} \int_{B_E(0,\rho_2(v_1)) \cap V_2}  \int_{B_E(0,\rho_3(v_1, v_2)) \cap V_3} \ldots \notag \\
&  \ldots \int_{B_E(0,\rho_{\iota-1}(v_1, \ldots, v_{\iota-2})) \cap V_{\iota-1}} \mathcal{L}^{n_{\iota}}(B_E(0,\rho_{\iota}(v_1, \ldots, v_{\iota-1})) \cap V_{\iota}) dv_{\iota-1} \ldots dv_3 dv_2 dv_1,
\end{align}
where $n_{\iota}=\dim(V_{\iota})$ and $B_E(x,r)=\{ y \in \G: |x-y|<r\}$, for every $x \in \G$ and $r>0$. On the other side, if $z$ is not necessarily the identity element, then Fubini's theorem yields that
\begin{equation*}
\begin{split}
&\mathcal{H}^n_{| \cdot |} (\B(z,1) \cap V)=\mathcal{H}^n_{| \cdot |} (B(z,1) \cap V)=\int_{B_E(z_1,\rho_1) \cap V_1} \int_{B_E(\Psi_2(v_1),\rho_2(v_1-z_1)) \cap V_2}   \ldots \\
&  \ldots \int_{B_E(\Psi_{\iota-1}(v_1, \ldots , v_{\iota-2}),\rho_{\iota-1}(v_1-z_1, v_2-\Psi_2(v_1), \ldots, v_{\iota-2}-\Psi_{\iota-2}(v_1, \ldots, v_{\iota-3}))) \cap V_{\iota-1}} \\
&\mathcal{L}^{n_{\iota}}(B_E(\Psi_{\iota}(v_1, \ldots, v_{\iota-1}),\rho_{\iota}(v_1-z_1, \ldots, v_{\iota-1}-\Psi_{\iota-1}(v_1, \ldots, v_{\iota-2})) \cap V_{\iota}) dv_{\iota-1} \ldots dv_2 dv_1.
\end{split}
\end{equation*}
As a consequence, we apply \cite[Theorem 6.3]{Mag22RS}, exploiting the convexity and the
symmetry of the Euclidean ball, hence getting that
\begin{align} \label{int}
\mathcal{H}^n_{| \cdot |}& (\B(z,1) \cap V) \leq \int_{B_E(z_1,\rho_1) \cap V_1} \int_{B_E(\Psi_2(v_1),\rho_2(v_1-z_1)) \cap V_2}  \ldots \notag \\
&\qquad \ldots \int_{B_E(\Psi_{\iota-1}(v_1, \ldots , v_{\iota-2}),\rho_{\iota-1}(v_1-z_1, v_2-\Psi_2(v_1), \ldots, v_{\iota-2}-\Psi_{\iota-2}(v_1, \ldots, v_{\iota-3})) \cap V_{\iota-1}}  \\
&   \qquad \qquad \mathcal{L}^{n_{\iota}}(B_E(0,\rho_{\iota}(v_1-z_1, \ldots, v_{\iota-1}-\Psi_{\iota-1}(v_1, \ldots, v_{\iota-2})) \cap V_{\iota}) dv_{\iota-1} \ldots dv_2 dv_1. \notag
\end{align}
Now, for every $i=1, \ldots, \iota-1$ we consider $V_i^{\perp}$ as the orthogonal complement of $V_i$ in $H_i$, hence
$V_i \oplus V_i^{\perp}=H_i$ and we consider, for $i=1, \ldots, \iota-1$, the splitting with respect to this direct sum
$$\Psi_i(v_1, \ldots, v_{i-1})=\zeta_i+w_i,$$ with $\zeta_i \in V_i$ and $w_i \in V_i^{\perp}$. For the sake of simplicity, we do not explicitly indicate the dependencies of $\zeta_i$ and $w_i$ on $v_1, \ldots, v_{i-1}$.
Notice that, for every $i=1, \ldots, \iota-1$, for every $w \in V_i^{\perp}$ and $v \in V_i$ we have
\begin{equation}\label{mon2}
|v| \leq |v-w|.
\end{equation}
Let us now continue from \eqref{int} and let us perform the change of variable $v_{\iota-1}'=v_{\iota-1}-\zeta_{\iota-1}$, getting
\begin{equation}\label{int2}
\begin{split}
&\mathcal{H}^n_{| \cdot |} (\B(z,1) \cap V) \leq \int_{B_E(z_1,\rho_1) \cap V_1} \int_{B_E(\Psi_2(v_1),\rho_2(v_1-z_1)) \cap V_2} 
 \ldots \\
&\ldots \int_{B_E(w_{\iota-1},\rho_{\iota-1}(v_1-z_1, v_2-\Psi_2(v_1), \ldots, v_{\iota-2}-\Psi_{\iota-2}(v_1, \ldots, v_{\iota-3})) \cap (V_{\iota-1}-\zeta_{\iota-1})} \\
& \qquad \qquad \qquad \mathcal{L}^{n_{\iota}}(B_E(0,\rho_{\iota}(v_1-z_1, \ldots, v_{\iota-1}'-w_{\iota-1})) \cap V_{\iota}) dv_{\iota-1}' \ldots dv_2 dv_1.
\end{split}
\end{equation}
Let us not collect three observations
\begin{itemize}
\item[(i)] $V_{\iota-1}-\zeta_{\iota-1}=V_{\iota-1}$;
\end{itemize}
By \eqref{mon2} we have $|v_{\iota-1}'-w_{\iota-1}|\geq |v_{\iota-1}'|$, then the
increasing monotonicity of $\rho_{\iota}$ with respect to each variable gives
\begin{itemize}
\item[(ii)] $ \rho_{\iota}(v_1-z_1, \ldots, v_{\iota-1}'-w_{\iota-1}) \leq \rho_{\iota}(v_1-z_1, \ldots, v_{\iota-1}') $
\item[(iii)] $B_E(w_{\iota-1},\rho_{\iota-1}(v_1-z_1,  \ldots, v_{\iota-2}-\Psi_{\iota-2}(v_1, \ldots, v_{\iota-3})) \cap V_{\iota-1} \subset B_E(0,\rho_{\iota-1}(v_1-z_1,  \ldots, v_{\iota-2}-\Psi_{\iota-2}(v_1, \ldots, v_{\iota-3})) \cap V_{\iota-1}.$
\end{itemize}
Thus, we continue from \eqref{int2} exploiting (i), (ii), and (iii) and we obtain the estimate
\begin{equation}
\begin{split}
&\mathcal{H}^n_{| \cdot |} (\B(z,1) \cap V) \leq \int_{B_E(z_1,\rho_1) \cap V_1} \int_{B_E(\Psi_2(v_1),\rho_2(v_1-z_1)) \cap V_2}   \ldots \\
& \ldots \int_{B_E(0,\rho_{\iota-1}(v_1-z_1, v_2-\Psi_2(v_1), \ldots, v_{\iota-2}-\Psi_{\iota-2}(v_1, \ldots, v_{\iota-3})) \cap V_{\iota-1}} \\
& \qquad \qquad \qquad  \mathcal{L}^{n_{\iota}}(B_E(0,\rho_{\iota}(v_1-z_1, \ldots, v_{\iota-1}')) \cap V_{\iota}) dv_{\iota-1}' \ldots dv_2 dv_1. 
\end{split}
\end{equation}
Now, we perform a second change of variable $v_{\iota-2}'=v_{\iota-2}-\zeta_{\iota-2}$ and we get
\begin{align}\label{int3}
\mathcal{H}^n_{| \cdot |}& (\B(z,1)  \cap V)  \leq \int_{B_E(z_1,\rho_1) \cap V_1} \int_{B_E(\Psi_2(v_1),\rho_2(v_1-z_1)) \cap V_2} \ldots \notag \\
& \qquad \qquad \ldots \int_{B_E(w_{\iota-2}, \rho_{\iota-2}(v_1-z_1, \ldots, v_{\iota-3}-\Psi_{\iota-3}(v_1, \ldots, v_{\iota-4})) \cap (V_{\iota-2}-\zeta_{\iota-2})}  \\
&   \int_{B_E(0,\rho_{\iota-1}(v_1-z_1, v_2-\Psi_2(v_1), \ldots, v_{\iota-2}'-w_{\iota-2})) \cap V_{\iota-1}} \notag \\
&  \qquad \qquad \mathcal{L}^{n_{\iota}}(B_E(0,\rho_{\iota}(v_1-z_1, \ldots, v_{\iota-2}'-w_{\iota-2}, v_{\iota-1}')) \cap V_{\iota}) dv_{\iota-1}' dv_{\iota-2}' \ldots dv_2 dv_1. \notag
\end{align}
Now, by adapting the observations (i), (ii) and (iii) to the index $\iota-2$, we can estimate \eqref{int3} as 
\begin{equation}\label{int4}
\begin{split}
&\mathcal{H}^n_{| \cdot |} (\B(z,1) \cap V) \leq \int_{B_E(z_1,\rho_1) \cap V_1} \int_{B_E(\Psi_2(v_1),\rho_2(v_1-z_1)) \cap V_2} \ldots \\
& \ldots \int_{B_E(0, \rho_{\iota-2}(v_1-z_1, \ldots, v_{\iota-3}-\Psi_{\iota-3}(v_1, \ldots, v_{\iota-4})) \cap V_{\iota-2}} \int_{B_E(0,\rho_{\iota-1}(v_1-z_1, v_2-\Psi_2(v_1), \ldots, v_{\iota-2}')) \cap V_{\iota-1}} \\
& \qquad \qquad \qquad  \mathcal{L}^{n_{\iota}}(B_E(0,\rho_{\iota}(v_1-z_1, \ldots, v_{\iota-2}', v_{\iota-1}')) \cap V_{\iota}) dv_{\iota-1}' dv_{\iota-2}' \ldots dv_2 dv_1. 
\end{split}
\end{equation}
We can go on iterating this procedure by considering the change of variable $v_i'=v_i-\zeta_i$ for $i=\iota-3, \ldots, 1$ and repeating considerations analogous to (i), (ii) and (iii) for the corresponding index, up to getting the estimate
getting 
\begin{align}\label{int7}
&\mathcal{H}^n_{| \cdot |} (\B(z,1) \cap V) \leq \int_{B_E(0,\rho_1) \cap V_1 } \int_{B_E(0,\rho_2(v_1')) \cap V_2}  \ldots \\
\int_{B_E(0,\rho_{\iota-1}(v_1', v_2', \ldots, v_{\iota-2}')) \cap V_{\iota-1}}& \mathcal{L}^{n_{\iota}}(B_E(0,\rho_{\iota}(v_1', v_2', \ldots, v_{\iota-2}', v_{\iota-1}')) \cap V_{\iota}) dv_{\iota-1}' dv_{\iota-2}' \ldots dv_2' dv_1'.  \notag
\end{align}
Now, by the comparison of \eqref{int7} and \eqref{zuguale0}, we have proved that for every $z \in \B(0,1)$ the equality
$$\mathcal{H}^n_{| \cdot |} (\B(z,1) \cap V) \leq \mathcal{H}^n_{| \cdot |} (\B(0,1) \cap V)$$
holds, and this gives the thesis.
\end{proof}

Obviously in the case $n=\q$ the spherical factor is constantly
equal to the volume of the unit ball $\cH_{|\cdot|}^q(\B(0,1))$.

\begin{Remark}
In relation to the proof of Theorem~\ref{T1}, an important point is the
fact that the functions $\phi_i$, see \eqref{eq:T_i} and \eqref{eq:rho_i},
can be directly defined, without a recursive process.
It is however interesting to notice that the sets 
$$
T_i=\{ (v_1, \ldots, v_{i-1}) \in V_1 \times \ldots \times V_{i-1}: \varphi(|v_1|, \ldots, |v_{i-1}|, 0, \ldots, 0)<1 \}
$$
defined in \eqref{eq:T_i} for $i=2,\ldots,\iota$, can be also written using 
$\rho_1$ and $\rho_{i-1}:T_{i-1}\to(0,+\infty)$ for $i=3,\ldots,\iota$.
Following the notation in the proof of Theorem~\ref{T1},
we have
\begin{equation}
T_2=\{v_1\in V_1: |v_1|<\rho_1\}
\end{equation}
and for $i=3,\ldots,\iota$ the equality
\begin{equation*}
 T_i=\{ (v_1, \ldots, v_{i-1}) \in V_1 \times \ldots \times V_{i-1} :  |v_{i-1}| < \rho_{i-1}(v_1, \ldots, v_{i-2}), \ (v_1, \ldots, v_{i-2}) \in T_{i-1} \}
 \end{equation*}
holds as well. 
Indeed, if $v_1\in T_2$, then $\ph(|v_1|,0,\ldots,0)<1$, hence 
\eqref{eq:rho_1} and the continuity of $\ph$ yield $|v_1|<\rho_1$. 
Conversely, if $|v_1|<\rho_1$, again \eqref{eq:rho_1} yields $t_0\ge0$
such that $|v_1|<t_0$ and $\ph(t_0,0,\ldots,0)<1$.
From the monotonicity of 
$\ph$, $\ph(|v_1|,0,\ldots,0)\le \ph(t_0,\ldots,0)<1$, hence
$v_1\in T_2$. Let us consider the remaining case $3\le i\le \iota$.
If $(v_1, \ldots, v_{i-1}) \in T_i$, by \eqref{eq:telle}
we notice that $(v_1, \ldots, v_{i-2}) \in T_{i-1}$, and by definition of $T_i$, we have
$$
|v_{i-1}| \in \{t \geq 0 : \varphi(|v_1|, \ldots, |v_{i-2}|,t,0 \ldots, 0)<1\}.
$$
In view of the continuity of $\ph$, we get
 $|v_{i-1}|<\rho_{i-1}(v_1, \ldots, v_{i-2}).$
If we now assume that $(v_1, \ldots, v_{i-2})\in T_{i-1}$ and $|v_{i-1}| < \rho_{i-1}(v_1, \ldots, v_{i-2})$, there exists $\tau$ such that $|v_{i-1}|<\tau< \rho_{i-1}(v_1, \ldots, v_{i-2})$ and 
 \begin{equation*}
 \varphi(|v_1|, |v_2|, \ldots, |v_{i-2}|, \tau, 0, \ldots, 0)<1.
 \end{equation*}
The same monotonicity of $\varphi$ ensures that
 $$\varphi(|v_1|, |v_2|, \ldots, |v_{i-2}|, |v_{i-1}|, 0, \ldots, 0) \leq \varphi(|v_1|, |v_2|, \ldots, |v_{i-2}|, \tau, 0, \ldots, 0)<1,$$ so that $(v_1, \ldots, v_{i-1}) \in T_i$ and this concludes the proof.
\end{Remark}

Now we introduce the class of homogeneous subspaces that makes 
multiradial distance rotationally invariant, according to
Theorem~\ref{teo:multirot}.

\begin{deff}\label{def:F_n1n2..}
Let us fix the integers $1 \leq n_1, n_2, \ldots, n_{\iota} \leq \q-1$.
We denote by $\cF_{n_1,\ldots,n_\iota}$ the family of all homogeneous subspaces $V=V_1 \oplus \ldots \oplus V_{\iota} \subset \G$ such
that $V_i \subset H_i$ and $\dim(V_i)=n_i$ for every $i=1, \ldots, \iota$. 
\end{deff}

\begin{teo}\label{teo:multirot}
	Let $\G$ be a homogeneous group of step $\iota$ and let $d$ be a multiradial distance. Then, for every $1 \leq n_1, n_2, \ldots, n_{\iota} \leq \q-1$, the distance $d$ is rotationally symmetric with respect to $\mathcal{F}_{n_1, \ldots, n_{\iota}}$. It means that the spherical factor $\beta_d$ becomes the geometric constant
	\[
	\omega(\cF_{n_1,\ldots,n_\iota})=\beta_d(V)=\mathcal{H}^n_{| \cdot |} (V \cap \B(0,1))
	\]
	with respect to all $V\in\cF_{n_1,\ldots,n_\iota}$, where
	$n=n_1+n_2+\cdots+n_\iota$. 
\end{teo}
\begin{proof}
Let $V=V_1 \oplus \ldots \oplus V_{\iota}$ and $W=W_1 \oplus \ldots \oplus W_{\iota}$ be two homogeneous subspaces of $\cF_{n_1,\ldots,n_\iota}$, namely 
$$\dim(V_i)=\dim(W_i)=n_i$$
for every $i=1, \ldots \iota$. Let us consider Euclidean isometries
$J_i:H_i\to H_i$ such that $J_i(V_i)=W_i$ and set for every $x=\sum_{i=1}^\iota x_i$ with $x_i \in H_i$ and $i=1, \ldots, \iota$,
the new isometry
\begin{equation}\label{eq:formT}
T:\G\to\G, \qquad T\left(\sum_{i=1}^{\iota}x_i \right)=\sum_{i=1}^{\iota}J_i(x_i).
\end{equation}
Indeed the layers $H_i$'s are all orthogonal to each other.
Since $T$ is invertible and the previous conditions give $J(V)=W$, we clearly have 
\begin{equation}\label{eq1}
T(\B(0,1) \cap V)=T(\B(0,1)) \cap W.
\end{equation}
So, if we prove that $T(\B(0,1))=B(0,1)$, then our claim follows by Theorem~\ref{T1}. Since  the inverse of $T$ has the same form
\eqref{eq:formT}, it is sufficient to show $T(\B(0,1)) \subset \B(0,1)$.
Due to the definition of multiradial distance and the fact that $T$ is an isometry, we get
 \begin{align*}
 T(\B(0,1))&= T\left( \left\{ \sum_{i=1}^{\iota}x_i \in \G : x_i \in H_i\ \text{for}\ i=1, \ldots \iota \ , \varphi(|x_1|, \ldots , |x_{\iota}|)\leq 1 \right\} \right)\\
  &= \left\{ \sum_{i=1}^{\iota}J_i(x_i) \in \G : x_i \in H_i\ \text{for}\ i=1, \ldots \iota \ , \varphi(|x_1|, \ldots , |x_{\iota}|)\leq 1 \right\}\\
  &= \left\{ \sum_{i=1}^{\iota}J_i(x_i) \in \G : x_i \in H_i\ \text{for}\ i=1, \ldots \iota \ , \varphi(|J_1(x_1)|, \ldots , |J_{\iota}(x_{\iota})|)\leq 1 \right\},
 \end{align*}
where the last set is contained in $\B(0,1)$. Thus, we get
\[
\mathcal{H}^n_{| \cdot |} (V \cap \B(0,1))=
\mathcal{H}^n_{| \cdot |} (W \cap \B(0,1)),
\]
concluding the proof.
\end{proof}

\begin{proof}[Proof of Theorem~\ref{T2}]
Our assumptions allow us to apply the area formula \eqref{intro:areageneral},
where $\T_p$ is the tangent subgroup to $\Sigma$ at $p$.
By \cite[Theorem 3.2.8]{FMS14}, for every $p \in \Sigma$, we have that 
$(\T_p, \V)$ is a couple of complementary subgroups. Since $(\W,\V)$ is also a couple of complementary subgroups, by \cite[Proposition 7.2]{Mag14}, it holds that
$$\dim(\T_p \cap H_i)=\dim(\W \cap H_i)=n_i$$
for every $i=1, \ldots, \iota$ for every $p \in \Sigma$. Hence, $\T_p \in \mathcal{F}_{n_1, \ldots, n_{\iota}}$ for every $p \in \Sigma$. Therefore our claim follows by Theorem \ref{teo:multirot}.
\end{proof}

Corollary~\ref{cor:levelsetmultiradial} is a direct consequence of Theorem~\ref{T2}. 

\begin{proof}[Proof of Corollary~\ref{cor:levelsetmultiradial}]
By our assumptions, we have $f^{-1}(0)\cap\Omega'=\Sigma\cap\Omega'=\Phi(A)$,
with $\Phi(w)=w\phi(w)$, and $\phi:A\to\M$ is continuously
intrinsically differentiable, by combining
\cite[Theorem~4.3.7]{CorniPhD} and \cite[Proposition~3.12]{CorMag23pr}.
We are in the position to apply Theorem~\ref{T2} to the graph mapping
$\Phi$, hence using both \eqref{eq:SN(B)} and the formula
\beq
J\Phi(w) = |\bV \wedge \bW| \ \frac{ J_Hf(\Phi(w))}{J_{\V}f(\Phi(w))},
\eeq
that is (85) of \cite{CorMag23pr}, the proof is concluded.
\end{proof}

\section{Spherical measure and centered Hausdorff measure}\label{sect:spherical_center}

In this section, we deal with the equality between spherical measure and centered Hausdorff measure.

\begin{teo}\label{teo:eqSC}
Let $\Sigma \subset \G$ be the intrinsic graph of a mapping $\phi:A \to \V$, where $A \subset \W$ is open and $(\W,\V)$ is a couple of complementary subgroups. If $\phi$ is continuously intrinsically differentiable and $d$ is multiradial, then
\begin{equation}\label{SN=CNmulti}
\mathcal{S}^\NN \res \Sigma=\mathcal{C}^\NN\res \Sigma,
\end{equation}
where $\NN$ is the Hausdorff dimension of $\Sigma$ and both $\mathcal{S}^\NN$ and $\mathcal{C}^\NN$ are constructed by $d$.
\end{teo}

\begin{proof}
By slightly modifying, actually simplifying, the proof of the upper-blow achieved in \cite[Theorem 1.1]{CorMag23pr}, we get 
\begin{equation}\label{eq:centdens}
\Theta^{*\NN}(\mu,x)=\limsup_{r \to 0^+} \frac{\mu(\B(x,r))}{r^\NN}=\mathcal{H}^m_{|\cdot|}(\B(0,1) \cap \T_x),
\end{equation}
for every $x \in \Sigma$, where $\Theta^{*\NN}(\mu,x)$ 
is the {\em upper $\NN$-density} of $\mu$ at $x$, \cite[Definition~1.7]{FSSC15}, and we have defined 
\begin{equation}
\mu(B)= \int_{\Phi^{-1}(B)} J\Phi(w)\  d \mathcal{H}_{|\cdot|}^\n (w)
\end{equation}
for every Borel set $B \subset \mathbb{G}$ and the Jacobian 
$J\Phi$ is introduced in \cite[Definition~7.1]{CorMag23pr}.
The reduction of the argument to prove the ``centered blow-up'' of 
\eqref{eq:centdens} can be noticed in looking at how the set of (66) in \cite{CorMag23pr} becomes simpler in the special case $x=y$. It precisely corresponds to the preimage of 
the metric unit ball with respect to the intrinsically rescaled
graph map. This mapping is going to converge to the graph map of the
intrinsic differential, whose image is exactly the tangent group.
Thus, we apply the differentiation theorem \cite[Theorem 3.1]{FSSC15}, that combined with \eqref{eq:centdens}, gives
\begin{equation}\label{eq:areacentrata}
		\int_{\Phi^{-1}(B)} J\Phi(n)\  d \mathcal{H}_{|\cdot|}^{m} (n)= \int_B \mathcal{H}^m_{|\cdot|}(\T_x \cap \B(0,1)) \ d \mathcal{C}^{M}(x),
\end{equation}
for every Borel set $B \subset \Sigma$.
Since $d$ is multiradial, Theorem~\ref{T1} holds, therefore \eqref{eq:areacentrata} and \eqref{intro:areageneral} lead us
to the equality \eqref{SN=CNmulti}.
\end{proof}

In the next theorem, we establish the equality between the spherical measure and the centered Hausdorff measure of a $(\G,\M)$-regular set of $\G$, when the metric unit ball of the homogeneous distance is 
a convex set.

\begin{teo}\label{th:CSSQ-P}
Let $\G$ and $\M$ be two stratified groups of topological dimensions $\q$ and $\p$, and of Hausdorff dimensions $Q$ and $P$, respectively. Let $\Sigma \subset \G$ be a $(\G,\M)$-regular set of $\G$ and suppose that $d$ is a homogeneous distance whose metric unit ball $\B(0,1)$ is convex.
Then the following equality holds 
	$$\mathcal{S}^{Q-P}\res \Sigma=\mathcal{C}^{Q-P}\res \Sigma.$$
\end{teo}

\begin{proof}
As in the proof of Corollary~\ref{cor:levelsetmultiradial}, 
$(\G,\M)$-regular sets of $\G$ are locally the intrinsic graphs
of maps which are continuously intrinsically differentiable.
Therefore it is not restrictive to assume that the whole
$\Sigma$ is an intrinsic graph exactly as in the 
assumptions of Theorem~\ref{teo:eqSC}.
We denote by $\Phi$ the graph mapping, whose image is $\Sigma$.
We consider the same measure $\mu$ in the proof of 
Theorem~\ref{teo:eqSC}, hence the same arguments
give 
\[
	\int_{\Phi^{-1}(B)} J\Phi(w)\  d \mathcal{H}_{|\cdot|}^{\q-\p} (w)= \int_B \mathcal{H}^{\q-\p}_{|\cdot|}(\T_x \cap \B(0,1)) \ d \mathcal{C}^{Q-P}(x).
\]
Let us notice that all the tangent subgroups $\T_x$ to $\Sigma$ are also normal subgroups, since they are kernels of h-differentials,
see (36) of \cite{CorMag23pr}.  
For this reason, we can apply Theorem~\ref{teo:ballconv} to
the area formula \eqref{intro:areageneral}, getting 
\[
	\int_{\Phi^{-1}(B)} J\Phi(w)\  d \mathcal{H}_{|\cdot|}^{\q-\p} (w)=
	\int_B \mathcal{H}^{\q-\p}_{|\cdot|}(\T_x \cap \B(0,1)) \ d \mathcal{S}^{Q-P}(x).
\]
Out claim immediately follows.
\end{proof}

\bibliography{References}
\bibliographystyle{plain}

\end{document}

%% file: def-Area.tex
\newcommand{\B}{\mathbb{B}}

\newcommand{\G}{\mathbb{G}}

\newcommand{\M}{\mathbb{M}}
\newcommand{\N}{\mathbb{N}}

\newcommand{\R}{\mathbb{R}}

\newcommand{\T}{\mathbb{T}}

\newcommand{\V}{\mathbb{V}}
\newcommand{\W}{\mathbb{W}}

\newcommand{\cC}{\mathcal{C}}

\newcommand{\cF}{\mathcal{F}}

\newcommand{\cH}{\mathcal{H}}

\newcommand{\cL}{\mathcal{L}}

\newcommand{\cS}{\mathcal{S}}

\newcommand{\ph}{\varphi}

\newcommand{\sm}{\setminus}

\newcommand{\diam}{\mbox{\rm diam}}

\newcommand{\Lie}{\mathrm{Lie}}

\newcommand{\q}{\mathrm q}
\newcommand{\p}{\mathrm p}

\newcommand{\n}{\mathrm n}

\newcommand{\NN}{\mathrm N}

\newcommand{\res}{\mbox{\LARGE{$\llcorner$}}}

\newcommand{\beqas}{\begin{eqnarray*}}
\newcommand{\eeqas}{\end{eqnarray*}}
\newcommand{\beqa}{\begin{eqnarray}}
\newcommand{\eeqa}{\end{eqnarray}}
\newcommand{\beq}{\begin{equation}}
\newcommand{\eeq}{\end{equation}}
\newcommand{\bce}{\begin{center}}
\newcommand{\ece}{\end{center}}

\newcommand{\set}[1]{\left\{ #1 \right\}}            

\newtheorem{The}{Theorem}[section]
\newtheorem{Lem}[The]{Lemma}
\newtheorem{Def}[The]{Definition}
\newtheorem{Pro}[The]{Proposition}
\newtheorem{Cor}[The]{Corollary}
\newtheorem{Exa}[The]{Example}
\newtheorem{Con}{Conjecure}

\newcommand{\bt}{\begin{The}}
\newcommand{\et}{\end{The}}
\newcommand{\bl}{\begin{Lem}}
\newcommand{\el}{\end{Lem}}
\newcommand{\bd}{\begin{Def}\rm}
\newcommand{\ed}{\end{Def}}
\newcommand{\br}{\begin{Rem}\rm}
\newcommand{\er}{\end{Rem}}
\newcommand{\bpr}{\begin{Pro}}
\newcommand{\epr}{\end{Pro}}
\newcommand{\bc}{\begin{Cor}}
\newcommand{\ec}{\end{Cor}}
\newcommand{\bj}{\begin{Con}}
\newcommand{\ej}{\end{Con}}
\newcommand{\bex}{\begin{Exa}}
\newcommand{\eex}{\end{Exa}}